\documentclass{amsart}
\usepackage{amsmath, amssymb, amsthm}
\usepackage{mathtools}
\usepackage{verbatim}
\usepackage{mathrsfs}

\usepackage[usenames,dvipsnames]{xcolor}

\usepackage{graphicx}
\usepackage[utf8]{inputenc}

\usepackage{tikz}
\usetikzlibrary{calc}

\usepackage{hyperref}
\usepackage[capitalize]{cleveref}
\crefname{ineq}{Ineq.}{inequalities}
\creflabelformat{ineq}{#2{\upshape(#1)}#3} 

\usepackage[normalem]{ulem}
\usepackage{enumerate}
\usepackage{enumitem}
\usepackage{todonotes}

\hypersetup{
	colorlinks   = true, 
	urlcolor     = blue, 
	linkcolor    = blue, 
	citecolor   = blue 
}

\makeatletter
\newtheorem*{rep@theorem}{\rep@title}
\newcommand{\newreptheorem}[2]{%
	\newenvironment{rep#1}[1]{%
		\def\rep@title{#2 \ref{##1}}%
		\begin{rep@theorem}}%
		{\end{rep@theorem}}}
\makeatother

\newtheorem{theorem}{Theorem}[section]
\newreptheorem{theorem}{Theorem}
\newtheorem{lemma}[theorem]{Lemma}
\newtheorem{proposition}[theorem]{Proposition}
\newtheorem{corollary}[theorem]{Corollary}
\newreptheorem{corollary}{Corollary}

\newtheorem{definition}[theorem]{Definition}

\newtheorem{claim}{Claim}
\newtheorem*{claim*}{Claim}
\newtheorem{Conjecture}{Conjecture}
\newtheorem*{Conjecture*}{Conjecture}

\theoremstyle{definition}
\newtheorem{example}[theorem]{Example}

\newtheorem{remark}[theorem]{Remark}

\theoremstyle{remark}

\newtheorem*{Note}{\bf Note}

\numberwithin{equation}{section}

\newcommand{\la}{\lambda}

\newcommand{\T}{\mathrm{T}}

\renewcommand{\bar}{\overline}

\newcommand{\cout}[1]{}

\definecolor{darkcyan}{rgb}{0. 0.65, 0.65}

\newcommand{\eps}{\varepsilon}

\newtheorem{Structural Stability Theorem}[theorem]{Structural Stability Theorem}

\def\bt{\begin{theorem}}
	\def\et{\end{theorem}}
\def\bd{\begin{definition}}
	\def\ed{\end{definition}}
\def\bl{\begin{lemma}}
	\def\el{\end{lemma}}

\def\be#1\ee{\begin{align}\begin{split} #1 \end{split}\end{align}}
\def\beq#1\eeq{\begin{align*}\begin{split} #1 \end{split}\end{align*}}

\begin{document}
\title{Rigidity of compact rank one symmetric spaces}
\author{Chris Connell$^\dagger$, Mitul Islam$^{\ddagger}$, Thang Nguyen$^{\ddagger\dagger}$ and Ralf Spatzier$^{\ddagger\ddagger}$}

\address{Department of Mathematics,
Indiana University, Bloomington, IN 47405}
\email{connell@iu.edu}
\address{Max Planck Institute for Mathematics in the Sciences, Inselstr. 22, 04103 Leipzig}
\email{mitul.islam@mis.mpg.de}
\address{Department of Mathematics, Florida State University, 
    Tallahassee, FL, 32304}
\email{tqn22@fsu.edu}
\address{Department of Mathematics, University of Michigan, 
    Ann Arbor, MI, 48109.}
\email{spatzier@umich.edu}
\thanks{$^\dagger$ Supported in part by Simons Foundation grant \#965245}
\thanks{$^{\ddagger}$ Supported in part by DFG Emmy Noether project 427903332 and DFG  project 338644254 (SPP 2026).}
\thanks{$^{\ddagger\dagger}$ Supported in part by grant Simons Travel Support for Mathematicians MPS-TSM-00002547.}
\thanks{$^{\ddagger\ddagger}$ Supported in part by NSF grant DMS 2003712.}

\date{\today}

\begin{abstract}
We consider rigidity properties of compact symmetric spaces $X$ with metric $g_0$ of rank one.  Suppose $g$ is another Riemannian metric on $X$ with sectional curvature $\kappa$ bounded by $0 \leq \kappa \leq 1$. If $g$ equals $g_0$ outside a convex proper subset of $X$, then $g$ is isometric with $g_0$.  We also exhibit examples of surfaces showing that the nonnegativity of the curvature is needed. Our main result complements earlier results on other symmetric spaces by Gromov and Schroeder-Ziller.
\end{abstract}

\maketitle

\section{Introduction}

Riemannian symmetric spaces possess amazing rigidity properties.  For one incarnation, consider compact 
convex subsets $D$ of a globally symmetric space $X$.   Under innocuous curvature assumptions, the metric on such $D$ is often fully determined by the metric outside $D$.   This theme was first pursued by Gromov for nonpositively curved symmetric spaces of rank at least two in \cite{BGS1985}, assuming that the unknown metric is also non-positively curved.  Schroeder and Ziller proved a more general version for nonpositively curved manifolds, only assuming that the metric is locally symmetric of rank at least three on a suitable open set \cite{Schroeder_1990}.  This was improved by Schroeder and Strake under a rank two  locally symmetric assumption \cite{Schroeder_1989}.

For nonnegatively curved symmetric spaces, Schroeder and Ziller also found a similar result under the assumption that rank is at least three or that rank is two with an additional hypothesis on the size of the set $D$ \cite{Schroeder_1990}.  Furthermore, Schroeder and Ziller proved similar results for globally symmetric spaces assuming that both the symmetric and the unknown metric have lower curvature bound one. 

Naturally, the question arises of what happens if we instead impose an upper curvature bound of one. When the sectional curvature $\kappa$ is trapped between $0 \leq \kappa \leq 1$ we provide the following positive answer in \cref{thm:main}.

Before stating the answer, we specify our definition of convexity to avoid confusion (as there are many distinct notions of convexity, even strict convexity,  in positive curvature).
A subset $D\subset M$ of a complete geodesic metric space $M$ is {\em s-convex} if every minimizing geodesic segment in $M$ joining $x,y\in D$ also belongs to $D$.

\begin{theorem}\label{thm:main}
    Let $(X, g_0)$ be a connected, simply connected, globally symmetric space of rank one with maximal curvature $1$.  Let $D$ be an s-convex closed subset of $(X,g_0)$. If $g_1$ is another smooth  Riemannian metric on $X$ with  $g_1|_{D^c}=g_0|_{D^c}$ and if $g_1$ has sectional curvatures $\kappa$ bounded by $0 \leq \kappa \leq 1$, then $g_1$ and $g_0$ are isometric. 
\end{theorem}

In particular, in the case of the sphere $X=S^n$, any metric that has curvature identically $1$ on any open neighborhood of a hemisphere must either be the round metric or have sectional curvatures outside the interval $[0,1]$ somewhere. 

 In their proof, Schroeder and Ziller rely on either the Rauch comparison theorem or the Toponogov comparison theorem (depending on the curvature bounds), and especially their rigidity properties.  
 For their result on metrics with sectional curvatures $\kappa \geq 1$, they critically use the existence of minimal curvature totally geodesic subspaces of large dimension. This argument fails for the upper curvature bound 1, since maximal curvature subspaces have small dimensions; indeed only dimension two for complex projective spaces. Instead, we realize that  $(X,g)$ has spherical rank at least one, i.e. that the first conjugate point along any geodesic is exactly at time $\pi$.  Then the result follows thanks to the spherical rank rigidity theorem of Shankar, Spatzier and Wilking \cite{SSW}.

\begin{Note}
We note that we don't need to use the full force of the spherical rank rigidity theorem.  We can just use 2-dimensional totally geodesic subspaces to show $(X,g)$ is a Blaschke manifold with injectivity radius $\pi$ and, by assumption,  $\kappa \leq 1$.  By a special case of the Blaschke conjecture \cite{RT} or \cite[Proposition 2.1]{SSW}, $(X,g)$ then is a compact rank one symmetric space.  Using spherical rank rigidity however, seems easier conceptually.

Furthermore, in lieu of spherical rank rigidity or the Blaschke conjecture, we can use the boundary rigidity results of \cite[Corollary 1.2]{StefanovUhlmannVasy21} if the set $D$ is convex in the sense that $\partial D$ is smooth and has positive second fundamental form. We note that their assumption of simplicity for the manifold follows from our \cref{{prop:hemisphere1}}.

\end{Note}

We will also prove a general version of \cref{thm:main} for CAT(1) metrics, see \cref{thm:extension_of_isom_inside_D}. Furthermore, from our considerations in \cref{sec:CAT_1} and \ref{sec:proof_of_main_thm}, we obtain the following:

\begin{corollary} \label{cor:shortclosed}
    Suppose $(X, g_0)$ is a connected, simply connected, globally symmetric space of rank one with maximal curvature $1$ and $D$ is an s-convex closed subset of $(X,g_0)$. If $g_1$ is a smooth Riemannian metric on $X$ whose sectional curvature $\kappa \leq 1$ and $g_1|_{D^c}=g_0|_{D^c}$, then the shortest closed geodesic in $(X,g_1)$ has length strictly less than $2\pi$ and $g_1$ has negative sectional curvature at some point in $D$. 
\end{corollary}

 Indeed, by our results in \cref{sec:proof_of_main_thm},  if the shortest closed geodesic in $g_1$ has length $\geq 2\pi$, then $g_1$ is a CAT(1) metric in contradiction to our \cref{thm:extension_of_isom_inside_D}.

We further prove that the curvature assumption $0 \leq \kappa \leq 1$ is actually needed, at least for surfaces.

\begin{theorem}[See \cref{eg:metric_on_surf_rev}]
\label{thm:surface_example}
Let $(S^2, g_0)$ be the 2-sphere with the round metric of sectional curvature $1$.  Then there is an s-convex subset $D$  of $S^2$ and a $C^{\infty}$ metric $g$ on $S^2$ which agrees with $g_0$ outside of $D$ and has sectional curvature bounded above by one but is not isometric to $g_0$. 
\end{theorem}

Our method for constructing these examples does not extend to higher dimensions.  Thus let us ask if \cref{thm:main} might hold with just an upper curvature bound in higher dimensions.  Note our conclusion in \cref{cor:shortclosed} about short closed geodesics as well as our \cref{rem:minmal} about small area minimal spheres.

One may further wonder if similar results hold for curvature bounds $-1 \leq \kappa \leq 0$.  For real hyperbolic space, one can adapt the arguments from Schroeder-Ziller. However, this fails for the other symmetric spaces (since the totally geodesic subspaces of maximal curvature do not intersect in a geodesic anymore). Still, we finish with 

\begin{Conjecture}
    Let $(X, g_0)$ be a globally symmetric space of rank one with minimal curvature $-1$, and $D$ a compact subset.  Then any Riemannian metric $g$ which coincides with $g_0$ outside of $D$ is isometric to $g_0$.
\end{Conjecture}

We note that if the diameter of $D$ is smaller than $\tanh^{-1}(\frac{\sqrt{2}}{2})$ then the conjecture holds by Toponogov and boundary rigidity theorem \cite[Corollary 1.2]{StefanovUhlmannVasy21}.  Indeed, such sets can be covered by triangles which fill in rigidly by Toponogov and the lower bound on curvature.

 \medskip

\noindent {\sl Acknowledgements:} We are grateful to K. Shankar for conversations about this project. CC, MI, and TN thank the Max Planck Institute for Mathematics in Bonn. MI and TN thank the University of Michigan, Institut Henri Poincar\'e, and LabEx CARMIN (ANR-10-LABX-59-01) for their support and hospitality. Further, MI thanks Universit\"{a}t Heidelberg while TN thanks Institut des Hautes \'Etudes Scientifiques.  

\section{Preliminaries}

\subsection{Definitions of convexity}

As there are many notions of convexity in Riemannian geometry, let us make precise definitions for the ones we will use. The first definition below can be interpreted as `weak' notion of convexity while the second one can be interpreted as `strong' notion of convexity.   

\begin{definition}
    A subset $D$ in a Riemannian manifold $(X,g)$ is called  {\em w-convex} if for every pair of points $x$ and $y$ in $D$, there exists a minimizing geodesic segment from $x$ to $y$ belonging to $D$.
\end{definition}

\begin{definition}
    A subset $D$ in a Riemannian manifold $(X,g)$ is called {\em s-convex} if for every pair of points $x$ and $y$ in $D$, every minimizing geodesic segment from $x$ to $y$ belongs to $D$.
\end{definition}

In addition, we will use the following notion of local convexity from Cheeger-Gromoll \cite{Cheeger_1972}. First we recall that for a Riemannian manifold $(X,g)$, the \emph{convexity radius} function is a function $r:X \to (0,\infty]$ such that, $r(p):=\sup\{s>0: $ any two points $q,q' \in B_s(p)$  are joined by a unique minimizing geodesic segment and the segment lies in $B_s(p) \}$.
\begin{definition}
    Let $(X,g)$ be a Riemannian manifold and let $r:M\to [0,\infty)$ be its convexity radius function. We say that a subset $D$ of $X$ is {\em  l-convex} if for every $p\in \bar D$, there exists $0<\eps(p)<r(p)$ such that $D\cap B_{\eps(p)}(p)$ has the following property: for every $x$ and $y$ in $D\cap B_{\eps(p)}(p)$, there is a unique minimizing geodesic segment between them and the segment is contained in $D\cap B_{\eps(p)}(p)$.
\end{definition}

\begin{lemma}  \label{lem:s->l convex}
Any $s$-convex set $D$ of a Riemannian manifold $(X,g)$ is $l$-convex.
\end{lemma} 

\begin{proof}
      Indeed, for any $p\in D$, set $\varepsilon(p):=0.9 r(p)$  where $r$ is the convexity radius function. Then, for any $x,y \in D \cap B_{\varepsilon(p)}(p)$, there is a unique minimizing $g_1$-geodesic segment joining $x$ and $y$ in $B_{\varepsilon(p)}(p)$ (since $B_{\varepsilon(p)}(p)$ is strongly convex, by definition of the convexity radius function). Since $D$ is s-convex for $g_1$, this minimizing $g_1$-geodesic also lies in $D$. 
\end{proof}

\subsection{Convex subsets of the positively curved globally symmetric spaces}

For the remainder of this section, consider a globally symmetric space $(X, g_0)$ with positive sectional curvature. We will prove some lemmas about the structure of convex sets.

\begin{lemma} \label{lem-baby}
    Let $D$ be a closed w-convex proper subset of 2-sphere $S^2$ with constant curvature $1$. Then $D$ is contained in a
closed hemisphere.
\end{lemma}

\begin{proof}

Define $F(z):=d(z,D)$ and let $R:=\sup_{z \in S^2}F(z)$.  As $D$ is proper and closed, $\{z : F(z) >0\}$ has non-empty interior. Thus $R>0$. By continuity of $F$, there exists $z_0$ such that $R=F(z_0)$. We note that in order to prove the lemma, it suffices to show that $R \geq \frac{\pi}{2}$. Indeed, if $R=F(z_0) \geq \frac{\pi}{2}$, then $D$ must be contained in the hemisphere which is opposite to $z_0$. 

So, for a contradiction, let us suppose that $R< \frac{\pi}{2}$. 
Then there exists $x \in D$ such that $d(z_0,x)=R$. Fix $\varepsilon$ such that $0 < \varepsilon<\frac{\pi}{2}-R$. Take the minimizing geodesic segment $[x,z_0]$ and extend it beyond $z_0$. Then pick $z$ on this geodesic segment (joining $z_0$ and $x$ extended beyond $z_0$) such that $d(z,z_0)=\varepsilon$. Then $F(z) \leq R$ and let $y \in D$ such that $d(z,y)=F(z)\le R$. By definition of $z_0$, we have $d(z_0,y)\ge R$. And thus $R\le d(z_0,y)\le R+\eps$. Apply the spherical law of cosine to the triangle with 3 vertices $z_0$, $z$, and $y$, the angle between $[z_0,x]$ and $[z_0,y]$ is bounded from below by $\pi- C$ where the angle $C$ satisfies
\[\cos C=\frac{\cos R - \cos \varepsilon \cos R}{\sin \eps \sin R}.\]

Since $\lim\limits_{\eps \to 0}\frac{\cos R - \cos \eps \cos R}{\sin \eps \sin R}=0$, then the angle between $[z_0,x]$ and $[z_0,y]$ is bigger than $\frac \pi 3$ for any sufficiently small $\eps >0$.

Note that $d(x,y) \leq d(x,z_0)+\varepsilon + d(z,y) \leq 2R+ \varepsilon <\pi,$ by choice of $\varepsilon$. Thus, there is a unique minimizing geodesic segment joining $x$ and $y$. By w-convexity, this lies in $D$. Let $u$ be its mid-point. 
Then $u\in D$. Since both $d(z_0,x)=R<\frac \pi 2$ and $d(z_0,y)\le R+\eps <\frac \pi 2$, we have an estimate $$d(z_0,u)<(1-\delta)\frac{d(z_0,x)+d(z_0,y)}{2}< (1-\delta)(R+\frac{\eps}{2}).$$ Here $\delta>0$ depends only on the angle between $[z_0,x]$ and $[z_0,y]$. As this angle is bounded below by $\pi/3$ for all $\eps$ sufficiently small, we can pick one $\delta >0$ for all these angles.  
Moreover, if $\eps$ is sufficiently small, then $(1-\delta)(R+\frac{\eps}{2})<R$. Thus, choosing $\eps$ sufficiently small, $F(z_0)=d(z_0,D) \leq d(z_0,u)<R$, which contradicts the definition of $R$ and $z_0$. \end{proof}

\begin{corollary}\label{cor:w-convex-of-sphere}
    Let $C$ be a open w-convex proper subset of a 2-sphere $S^2$ with constant curvature $1$. Then $C$ is contained in an open hemisphere.
\end{corollary}

\begin{proof}
    The closure $\bar C$ is still a w-convex proper subset of $S^2$. Indeed, we only need to prove that $\bar C$ is proper. If that is not the case, let $z \in \partial C$ and take a sufficiently small open w-convex neighborhood $U$ of $z$ in $S^2$. In $U$, we can find three points that lie in $C$ whose convex hull contains $z$, and thus $z \in C$. This implies that $C$ contains $\bar C$ and hence $C$ is not proper, a contradiction. Now by the previous \cref{lem-baby}, $\bar C$ is contained in a closed hemisphere. Since $C$ is open, $C$ must be in an open hemisphere.
\end{proof}

\begin{corollary}\label{cor:s-convex-of-sphere}
    Let $D$ be a closed s-convex proper subset of a 2-sphere $S^2$ with constant curvature $1$. Then $D$ is contained in an open hemisphere.
\end{corollary}

\begin{proof}
Define $F(z):=d(z,D)$ and let $R:=\sup_{z \in S^2}F(z)$. It suffices to prove that $R> \frac \pi 2$.

Since $D$ is s-convex, $D$ is also w-convex. By \cref{lem-baby}, we obtain that $R\ge \frac \pi 2$. Now suppose that $R=\frac \pi 2$. Let $z_0$ be a point that realizes the maximum value of $F$, i.e. $F(z_0)=\frac \pi 2$. Let $C$ be the great circle of distance $\frac \pi 2$ from $z_0$.  Since $D$ is closed, there exist points in $D$ of distance $\frac \pi 2$ from $z_0$. We let $y$ and $z$ be on $C$ such that the geodesic segment $[y,z]\subset C$ is the maximum segment in $C$ contained in $D$. Note that a priori $y$ could equal $z$.  If the length of this segment is at least $\pi$ then the entire set $C$ is contained in $D$. By s-convexity, $D=S^2$, a contradiction. Thus we we assume that the length of the segment $[y,z]$ is strictly smaller than $\pi$. We let $y_1$ and $z_1$ in $C-D$ of distance $\pi$ from each other such that one of the geodesic segments,  $[y_1,z_1]\subset C$ say,  contains $[y,z]$. Let $w$ be the midpoint of the geodesic segment from $y_1$ to $z_1$ disjoint from $[y,z]$. Let $u_t$ be the point on the geodesic segment from $z_0$ to $w$ and  of distance $t$ from $z_0$. We claim that $D$ is disjoint from a closed ball of radius $\frac \pi 2$ centered at $u_t$ for $t$ sufficiently small. Suppose this is not the case. Then for every $t$, the set $\bar{B_{u_t}(\frac \pi 2)}\setminus\bar{B_{z_0}(\frac \pi 2)}$ always contains a point $p_t$ in $D$. After passing to a subsequence of $p_t$  of the points and taking a limit, we get a point $p$ lying in the geodesic segment from $y_1$ to $z_1$ in $C$ not containing $[y,z]$. Since $D$ is closed, $p\in D$. This contradicts with the maximality of $[y,z]$ in $C$.
\end{proof}

\begin{proposition}\label{prop:hemisphere1}
    Let $D$ be a closed s-convex proper subset of $(X,g_0)$ with non-empty interior. Then the intersection of $D$ with every totally geodesic sphere of curvature $1$ is contained in an open hemisphere. 
\end{proposition}

\begin{proof}
    The set $D$ intersects  any totally geodesic 2-sphere of curvature $1$ in a closed s-convex set. By \cref{cor:s-convex-of-sphere}, such a closed s-convex set is either proper, and hence contained in an open hemisphere, or it is the entire sphere. 
    
    Now suppose that there is a totally geodesic 2-sphere $S$ contained in $D$. If $X$ is an $n$-sphere then $D$ is equal to $X$ by  s-convexity. We thus assume that $X$ is not a sphere. We first observe that $S$ is contained in a maximal totally geodesic sphere $S'$ of constant curvature $1$. By s-convexity, the sphere $S'$ is also contained in $D$. Let $p$ be an interior point of $D$. There exists a point $q$ in $S'$ of distance $\pi$ from $p$. Otherwise, the sphere $S'$ is contained in a ball of radius smaller than $\pi$ around $p$, and thus would be homotopic to $\{p\}$. This would be in contradiction to the fact that $S'$ defines a nontrivial  homology class in $H_*(X)$ \cite[Theorem 3.42]{Besse78}. By s-convexity, all geodesic segments from $p$ to $q$ are contained in $D$. In particular, the maximal totally geodesic sphere of curvature $1$ containing $p$ and $q$ are in $D$. Let $\gamma$ be a closed geodesic in this sphere containing $p$ and $q$. Then $\gamma$ contains interior points of $D$. As $s$-convex sets are always $l$-convex by \cref{lem:s->l convex}, this contradicts with \cite[Theorem 1.10]{Cheeger_1972}.
\end{proof}

\begin{proposition}\label{lem:hemisphere2}
    Let $C$ be a maximal open w-convex proper subset of $(X,g_0)$. Then the intersection of $C$ with any totally geodesic 2-sphere of curvature $1$ is contained in an open hemisphere.
\end{proposition}

\begin{proof}
    By \cref{cor:w-convex-of-sphere}, because of the convexity, the intersection of $C$ and each totally geodesic 2-sphere $S$ of $g_0$ of curvature $1$ is either a subset of a hemisphere or all of $S$. In the latter case, we claim that $C = X$. This will imply the lemma. 
    
    To prove the claim,  consider any totally geodesic $S^2$ (if $X$ is a sphere) or $\mathbb RP^2$ (if $X$ is a compact symmetric space other than a sphere) containing a closed geodesic of  $S$. The set $C$ intersects this $S^2$ or $\mathbb RP^2$ in an open convex neighborhood of a closed geodesic. Then $C$ contains this entire $S^2$ or $\mathbb RP^2$.

    Now, let  $y\in X$  be arbitrary. Schroeder and Ziller in their proof of Lemma 4b that show there is a chain of finitely many totally geodesic $S^2$s or $\mathbb RP^2$s in which the first $S^2$ or $\mathbb RP^2$ intersect $S$ in a closed geodesic, and any two consecutive $S^2$ or $\mathbb RP^2$ in the chain intersect in a closed geodesic, and the last $S^2$ or $\mathbb RP^2$s contains the point $y$,  cf. \cite[p. 155]{Schroeder_1990}. Note that this finite chain is easy to obtain for $X$ the round sphere. 
     It follows that $y$ is contained in $C$. Thus, $C$ is entirely $X$, a contradiction. 
    
    Therefore, the intersection of $C$ and each totally geodesic $S^2$ is contained in a hemisphere.
\end{proof}

\section{Rigidity of CAT(1) metrics on symmetric spaces}
\label{sec:CAT_1}

In this section, we prove the CAT(1) version of our results which we will use in the next section to prove \cref{thm:main}.

\begin{theorem} 
\label{thm:extension_of_isom_inside_D}
Let $(X, g_0)$ be a connected, simply connected globally symmetric space of rank one and maximal sectional curvature $1$. Let $D$ be either
\begin{itemize}
    \item a closed subset of a proper maximal w-convex open subset of $(X,g_0)$, or
    \item a closed s-convex proper subset of $(X,g_0)$. 
\end{itemize} 
Let $(Y,g)$ be a complete simply connected smooth Riemannian CAT(1) manifold with sectional curvatures bounded above by one and $\dim(Y)=\dim(X)$. If $f:X\setminus D\to Y$ is an isometric embedding then $f$ extends to an isometric embedding from $X$ to $Y$. In particular, if $Y$ is connected then $(Y,g)$ is isometric to $(X,g_0)$.
\end{theorem}

Now we begin the proof of \cref{thm:extension_of_isom_inside_D} which will work under either condition on the convex set in the theorem. Without loss of generality, assume that $Y$ is also connected. Note that it suffices to work with the assumption that the interior of $D$ is non-empty.  Otherwise, we can find an obvious isometric extension of $f$ to $D$.  

The main idea is showing that $(Y,g)$ has higher spherical rank, i.e. that any unit speed geodesic has a conjugate point at distance $\pi$. This follows immediately from the next lemma.

\begin{lemma} \label{lem:totgeo}
    For every point $p\in Y\setminus f(D^c)$ and $q\in f(D^c)$, and for every geodesic segment $\gamma$ from $p$ to $q$, there exists a $g$-totally geodesic sphere of sectional curvature $1$ containing $\gamma$.
\end{lemma}

\begin{proof} In the case that $D$ is a closed subset of an open w-convex set, without loss of generality, we can replace $D$ by its convex hull in $X$ since this convex hull is still contained in the same open w-convex set.

Take $q\in f(D^c)$ and look at a geodesic segment $\gamma$ from $q$ to $p$. Let $v$ be the unit tangent vector to this geodesic segment at $q$. Since $\dim X=\dim Y$, a beginning segment $\gamma_0$ of this geodesic segment belongs to $f(D^c)$. Let $w=Df^{-1}(v)$. There exists a totally geodesic sphere $S_w$ of sectional curvature $1$ in $(X,g_0)$ tangent to $w$ at $f^{-1}(q)$. Note that by \cref{prop:hemisphere1} and \cref{lem:hemisphere2}, the intersection $D\cap S_w$ is contained in an open hemisphere. Since $f$ is an isometry on $D^c$, $f^{-1}(\gamma_0)$ is a geodesic segment. By moving $q$ along $\gamma_0$, $f^{-1}(q)$ moves closer toward $D$. Without loss of generality, we assume that $f^{-1}(q)$ is also contained in an open hemisphere of $S_w$ which contains $S_w\cap D$.

\begin{claim*}
There exists a geodesic triangle $\Delta\subset S_w\setminus D$ that lies in an open hemisphere containing $D\cap S_w$ and $f^{-1}(q)$. 
\end{claim*}
\begin{proof}[Proof of Claim]
By the  above, we know that $D \cap S_w$ and $f^{-1}(q)$ are contained in an open hemisphere of $S_w$. Consider three points near the equator of this open hemisphere, on the side of $D\cap S_w$. We take the $g$-geodesic triangle spanned by these three points. Eventually, bringing these three points close enough to the equator, we can assume that the entire geodesic triangle lies outside $D$ (since $D \cap S_w$ is a compact subset of the open hemisphere, in both cases of the theorem).  
\end{proof}

Continuing with the proof of Lemma \ref{lem:totgeo}, note that the triangle $\Delta$ has the same edge lengths and angles as its comparison triangle in the 2-sphere of sectional curvature $1$. So, the perimeter of $\Delta$ is smaller than $2\pi$. Because $f$ is an isometry, $f(\Delta)$ has the same properties.  Since $g$ is CAT(1) by assumption, the rigidity part of Rauch's triangle comparison theorem \cite[Proposition 3.16]{Ballmann2004} implies that there exists a totally geodesic filling $T_\Delta'$ of $f(\Delta)$ in the $g$ metric which is isometric to the totally geodesic filling $T_\Delta \subset S_w$ of the spherical triangle $\Delta$ in the $g_0$ metric.

 Define $S_v':=f(S_w\setminus \T_{\Delta}) \cup T_{\Delta'}$. We claim that $S_v'$ is a $g$-totally geodesic 2-sphere of radius $1$. To prove this, it suffices to show that $f(T_{\Delta}\setminus D)$ and $T_{\Delta'}$ agree with each other in a neighborhood of $f(\Delta)$. This is because $f(S_w\setminus D)$ and $T_{\Delta}'$ are both $g$-totally geodesic subsets and so, if $f(T_{\Delta}\setminus D)$ and $T_{\Delta'}$ agree near $f(\Delta)$, then $S_v'$ will also be $g$-totally geodesic. Moreover, $S_v'$ is a topological sphere by construction.  
 
 We now establish the claim that $f(T_{\Delta}\setminus D)$ and $T_{\Delta'}$ agree with each other in a neighborhood of $f(\Delta)$. Indeed,  suppose $x\in f(T_\Delta\setminus D)$ lies near its boundary $f(\Delta)$, and let $y=f^{-1}(x)$. Then there exists a $g_0$-geodesic segment $l$ outside of $D$ containing $y$ with two endpoints $A,B$ in $\Delta$.  Since $\Delta$ lies in an open hemisphere by the  Claim above, the minimal geodesic between $A$ and $B$ is unique.  The image $f(l)$ is the unique minimal  $g$-geodesic segment containing $x$ and $f(A)$ and $f(B)$
 in $f(T_\Delta\setminus D)$. Since $T_{\Delta'}$ is totally geodesic and $g$ agrees with $f_*g_0$ on $f(D^c)$, the segment $f(l)$ is also contained in $T_{\Delta'}$. In particular, $x\in T_{\Delta'}$. An analogous argument shows that a neighborhood of $f(\Delta)$ in $T_{\Delta'}$ belongs to $f(T_{\Delta}\setminus D)$.

Finally, the geodesic segment $\gamma$ from $q$ to $p$ has the initial vector tangent to $S_v'$, hence is contained entirely in $S_v'$.
\end{proof}

\begin{corollary}\label{cor:higherrank-inside}
    For every $p\in f(D^c)$,  every $v \in T_pY$ has higher spherical rank.
\end{corollary}

\begin{proof}

    Let $\gamma$ be the geodesic with initial vector $v$. If $\gamma$ is contained entirely in $f(D^c)$ then $\gamma$ is the image of a closed geodesic of length $2\pi$ in $X\setminus D$. Since $f$ is isometric, some conjugate point of $p$ along $\gamma$ has distance $\pi$ from $p$, and thus $v$ has higher spherical rank.

    Now  assume that $\gamma$ intersects $f(D^c)^c$ at a point $q$. By \cref{lem:totgeo}, there exists a totally geodesic sphere $S$ of curvature $1$ containing the segment of $\gamma$ from $p$ to $q$. It follows that the entire $\gamma$ is contained in $S$. Therefore $v$ has higher spherical rank. 
\end{proof}

\begin{lemma}\label{lem:higherrank-outside}
    For every $p\in Y\setminus f(D^c)$, the set of  $v \in T_pY$ with higher spherical rank w.r.t $g$ is both open and closed.   
\end{lemma}
\begin{proof}

    We in fact prove that the set $\mathcal V$ of vectors $v\in T_pY$ such that the geodesic ray with initial vector $v$ intersects $f(D^c)$ is both open and closed. By  \cref{lem:totgeo}, such a vector has higher spherical rank.

    First we prove that $\mathcal V$ is open. Fix $p$ as in the statement of the lemma. Let $v_q$ be an initial vector of a geodesic segment joining $q \in f(D^c)$ and $p$. Let $N_p$ denote the sphere of radius $d(p,q)$ around the point $p$. Since $\dim(Y)=\dim(X)$, $f(D^c)$ is open in $Y$. Thus a neighborhood of $q$ is contained in $f(D^c)$. If $q' \in N_p$ near $q$, then there exists a unique vector $v' \in T_pY$ such that the geodesic ray initially tangent to $v'$ passes through $q'$. This proves that $\mathcal V$ is open. 

    Next, we prove that $\mathcal V$ is closed. Suppose that there exists a sequence of vectors $v_n \in T_p M $ converging to $v \in T_p M$ such that
\begin{itemize}
    \item $\gamma_{v_n}$ contains an interior point of $f(D^c)$, {\bf and} 
    \item $\gamma_v$ does not contain any interior point of $f(D^c)$.
\end{itemize}

Let $S_{v_n}$ be a totally geodesic sphere containing $\gamma_{v_n}$, as in \cref{lem:totgeo}. The sequence $S_{v_n}$ converges to a totally geodesic subset $S_v$ which contains the geodesic passing through $p$ in the direction of $v$.

Let $\phi _n$ be isometries $S^2 \mapsto  S_{v_n}$. 
Let $\phi: S^2 \mapsto Y $ be an accumulation point of the $\phi_n$.
We claim that they converge to an isometry $\phi: S^2 \mapsto S_v$.  Indeed, suppose $a_n, b_n \in S_{v_n}$ are sequences of points converging to $a$ and $b$ respectively,  with distance $d(a_n, b_n) = \pi$ in $Y$.  Then also the distance $d(a,b)$ in $Y$ is $\pi$. 
Finally, suppose $\phi$ is not an isometry.  Then there is a (closed) geodesic $\delta$ in $S^2$ on which $\phi$ is not an isometry.  This however is impossible since $\phi$ preserves 
the distance $\pi$ for opposite points of $\delta$.

We let $S_{v_n}^*$ be the totally geodesic 2-sphere in $X$ containing $f^{-1}(S_{v_n}\cap f(D^c))$. Passing to a subsequence, we assume that $S_{v_n}^*$ converges to a totally geodesic 2-sphere $S_v^*$.

We observe that $S_v^*$ is isometric to $S_v$, as they are both 2-spheres of curvature $1$ and $f$ maps $S_v^*\cap D^c$ isometrically to $f(D^c)\cap S_v$. It follows that the restriction of $f$ to $S_v^*\cap D^c$ extends to an isometry $f_S:S_v^*\to S_v$. The pull-back $f_S^{-1}(\gamma_v)$ is a geodesic in $S_v^*$. Since $S_v^*\cap D$ is contained in an open hemisphere, $f_S^{-1}(\gamma_v)$ intersects $S_v^*\cap D^c$ nontrivially. Pushing forward, we get that $\gamma_v$ contains a point in $S_v\cap f(D^c)$. In other words, $\gamma_v$ contains an interior point of $f(D^c)$. This is a contradiction.
\end{proof}
\begin{proof}[Proof of \cref{thm:extension_of_isom_inside_D}]
    For every $p\in Y\setminus f(D^c)$, pick $q\in f(D^c)$, then the initial vector $v$ of the geodesic segment from $p$ to $q$ has higher spherical rank by \cref{lem:totgeo}. Hence the set of higher spherical rank vectors in $T_pY$ is non-empty.  By \cref{lem:higherrank-outside}, every vector tangent to $Y$ at $p$ has higher spherical rank. Combining with \cref{cor:higherrank-inside}, every tangent vector of $Y$ has higher spherical rank.
    
    By spherical rank rigidity \cite[Theorem 1]{SSW}, $Y$ is a symmetric space of compact type. We show that $f$ extends to an isometry from $X$ to $Y$. First we note that $(Y,g)$ is isometric to $(X,g_0)$. Indeed, $X$ and $Y$ are both connected, simply connected symmetric spaces of compact type of the same dimension that are isometric on an open set. Thus, $X$ and $Y$ are isometric. By homogeneity, there is an isometry $\bar{f}$ agreeing with $f$ on $D^c$. So $f$ extends to an isometry from $X$ to $Y$.
\end{proof}

\section{Proof of Theorem 1.1}\label{sec:CAT(1)}
\label{sec:proof_of_main_thm}

In this section, we will reduce  \cref{thm:main} to   \cref{{thm:extension_of_isom_inside_D}}. In fact, we have a slightly more general version of \cref{thm:main} as follows.

\begin{theorem}\label{thm:main-embedding}
    Let $(X, g_0)$ be a connected, simply connected globally symmetric space of rank one with maximal curvature $1$.  Let $D$ be a closed s-convex subset of $(X,g_0)$. Let $(Y,g_1)$ be a connected, simply connected Riemannian manifold of sectional curvature $0\le \kappa\le 1$ and of dimension $\dim Y=\dim X$. If $f:X\setminus D\to Y$ is an isometric embedding, then $f$ extends to an isometry $f:X\to Y$.
\end{theorem}

\cref{thm:main-embedding} will follow from \cref{thm:extension_of_isom_inside_D} because of the following proposition.
\begin{proposition}
    \label{prop:metric_g1_is_CAT_1}
    Under the assumptions of \cref{thm:main-embedding}, the manifold $(Y,g_1)$ is a CAT(1) metric space.
\end{proposition}

\begin{proof}
By \cite[Part II, Proposition 4.16]{BH},  it suffices to show that the injectivity radius of $Y$ is at least $\pi$. Suppose, on the contrary, that this is not the case. We first claim:

    \begin{claim}
    \label{claim:closed_geod_length_2_times_injrad}
     There exists a closed geodesic $\gamma$ for the $g_1$-metric of length two times the injectivity radius.
     \end{claim}
     \begin{proof}[Proof of Claim]This is standard, cf. \cite[Chapter 13, Proposition 2.13]{doCarmo92} where this is proved under a positive curvature assumption, to ensure compactness of the manifold.  Indeed, let $p\in Y$ be a point that realizes the injectivity radius. We let $C(p)$ denote the cut locus of $p$. Since $Y$ is compact, $C(p)$ is also compact. Thus there exists $q\in C(p)$ such that $d(p,q)=d(p,C(p))$, which equals to the injectivity radius of $Y$. If $q$ is conjugate to $p$ then by Rauch's comparison theorem, $d(p,q)\ge \pi$; a contradiction because injectivity radius of $Y$ is less than $\pi$ by assumption. Thus $q$ is not conjugate to $p$. It follows that there exist exactly two geodesic segments $\gamma_1$ and $\gamma_2$ from $p$ to $q$ with the property that $\gamma_1'(l)=-\gamma_2'(l)$, where $l=d(p,q)$ \cite[Proposition 2.12]{doCarmo92}. On the other hand, by the choices of $p$ and $q$, $p$ is in the cut locus $C(q)$ of $q$ and $d(p,q)=d(q,C(q))$. It follows that $\gamma_1'(0)=-\gamma_2'(0)$. Then the union of $\gamma_1$ and $\gamma_2$ is a closed geodesic.
     \end{proof}

    Now we observe that the set $f(D^c)^c$ is s-convex for the metric $g_1$. Suppose, if possible, that there exist $x,y \in f(D^c)^c$ and a minimizing $g_1$-geodesic $\sigma$ between $x$ and $y$ that intersects $f(D^c)^c$ only at its endpoints. Then $\sigma-\{x,y\}$ lies in $f(D^c)$ and thus $f^*\sigma$ is a minimizing $g_0$-geodesic. As $D$ is s-convex for $g_0$, $f^*\sigma \subset D$, a contradiction. 

    Next we claim that $f(D^c)^c$ is $l$-convex. 
    Indeed, for any $p\in f(D^c)^c$, set $\varepsilon(p):=0.9 r(p)$  where $r$ is the convexity radius function. Then, for any $x,y \in f(D^c)^c \cap B_{\varepsilon(p)}(p)$, there is a unique minimizing $g_1$-geodesic segment joining $x$ and $y$ in $B_{\varepsilon(p)}(p)$ (since $B_{\varepsilon(p)}(p)$ is strongly convex, by definition of the convexity radius function). Since $f(D^c)^c$ is s-convex for $g_1$, this minimizing $g_1$-geodesic also lies in $f(D^c)^c$.

    \begin{claim}\label{claim:contained}
        If $\gamma$ is a closed geodesic as in \cref{claim:closed_geod_length_2_times_injrad}, then $\gamma$ cannot be contained entirely in $f(D^c)^c$. 
    \end{claim}
    \begin{proof}[Proof of \cref{claim:contained}]
    Suppose $\gamma$ is contained in  the l-convex set $f(D^c)^c$.  By \cite[Theorem 1.10]{Cheeger_1972}, the distance function from $\gamma$ to the boundary $\partial ( f(D^c)^c)$ is weakly convex.  Since $\gamma$ is periodic, this distance function must be constant.  Then, 
by \cite[Theorem 1.10]{Cheeger_1972}, there is an isometric copy of $\gamma \times [0,a]$ in $f(D^c)^c$ such that $\gamma \times \{0\}$ is the closed geodesic $\gamma$ itself, and $\gamma \times \{a\}$ is a closed geodesic in the metric $g_1$ that  lies in $\partial f(D^c)^c$. Since $\dim(Y)=\dim(X)$ and $f$ is isometric, $\partial f(D^c)^c$ is isometric to $\partial D$ with the standard metric. But there are no closed geodesics in $(X,g_0)$ of length shorter than $2 \pi$ (hence, also not in $\bar {f(D^c)}$). Thus $\partial f(D^c)^c$ does not contain any closed geodesic of length shorter than $2\pi$. Thus, $\gamma$ is not contained entirely in $f(D^c)^c$.
    \end{proof}

    Let $\gamma$ be a closed geodesic as in \cref{claim:closed_geod_length_2_times_injrad}. Since $\gamma$ has length less than $2\pi$ by assumption, it cannot be entirely contained in $f(D^c)$ since $f(D^c)$ is isometric to $(X\setminus D,g_0)$. By \cref{claim:contained}, $\gamma$ is also not entirely contained in $f(D^c)^c$. Thus, since $\dim(Y)=\dim(X)$, we can pick a point $p\in \partial f(D^c)^c\cap \gamma$ and also choose a parametrization of $\gamma$ such that $\gamma$ lies in $f(D^c)$ after $p$ and in $f(D^c)^c$ before $p$.   As $\gamma$ is a closed geodesic, it must re-enter $f(D^c)^c$ and we let $q_1$ be the first point after $p$ where $\gamma$ re-enters $f(D^c)^c$. Moreover, let $p_1:=p$ and $\gamma_1$ be the part of $\gamma$ between $p_1$ and $q_1$. 
    
    \begin{claim}
        \label{claim:length_gamma1_greater_than_pi}
        The length of $\gamma_1$ is strictly greater than $\pi$.
    \end{claim}
    \begin{proof}
        We claim that there exists a totally geodesic sphere $S$ of curvature $1$ in $(X,g_0)$ such that $\gamma_1$ is contained in $f(S\setminus D)$. Indeed, let $v=Df^{-1}(\gamma'(p))$, and let $S$ be a totally geodesic sphere of curvature $1$ in $(X,g_0)$ tangent to $v$. Since $f$ is isometric, the geodesic $\gamma_1$, between $p_1$ and $q_1$, is contained in $f(S\setminus D)$. 
    
        Then $f^{-1}(\gamma_1)$ is a geodesic in the 2-sphere $S$ of curvature $1$ and $f^{-1}(\gamma_1) \subset S\setminus D$. By \cref{prop:hemisphere1}, $D\cap S$ is a closed subset of an open hemisphere. It follows that
     the length of the arc $f^{-1}(\gamma_1)$ must be strictly greater than $\pi$. Since $f$ is isometric, the length of $\gamma_1$ is also strictly greater than $\pi$.
    \end{proof}

    We will show that there is a smooth variation $\{\gamma_s\}$ of $\gamma$ through closed curves, not necessarily geodesic, so that each curve in the variation has $g_1$-length shorter than $\gamma$. Indeed, there is a smooth variation $\{V_s\}$ of $f^{-1}(\gamma_1)$ inside $S\setminus D$ such that every curve in $\{V_s\}$ has shorter length than $f^{-1}(\gamma_1)$ (Here, $S$ is a totally geodesic sphere in $X$ of curvature $1$, as constructed in the proof of \cref{claim:length_gamma1_greater_than_pi}). Pushing this variation forward by $f$ we obtain a variation $\{W_s\}$ of $\gamma_1$ that has shorter length than $\gamma_1$. Let $\gamma_2=\gamma \setminus \gamma_1$. Then $W_s\cup \gamma_2$ is a piecewise smooth variation of $\gamma$ where each curve has shorter length than $\gamma$. Smoothing this variation, we obtain a smooth variation $\{\gamma_s\}$ of $\gamma$, through closed curves, where each curve has $g_1$-length shorter than $\gamma$.

    We will now arrive at a contradiction by an argument following \cite[Chapter 13, Proposition 3.4, page 282]{doCarmo92}. Let $p_2$ and $q_2$ be two points of furthest distance apart on $\gamma$. By Rauch's comparison theorem, $q_2$ and $p_2$ are not conjugate points of each other. 
    Let $\{\gamma_s\}$ be the variation we constructed in the previous paragraph. Let $q_s$ be a point of  $\gamma_s$ at maximum distance from $p_2$. Let $p_s$ be the point on $\gamma_s$ that is at maximum distance from $q_s$. Since $d(p_s, q_s) <\frac{1}{2} \text{length}(\gamma)=\text{injrad}(Y)$,  there exists a unique minimizing geodesic segment $\alpha_s$  joining $p_s$ and $q_s$. We know that $p_s \to p_2$ and $q_s \to q_2$. We claim that $\alpha_s'(p_s)$ is orthogonal to $\gamma_s$ for all $s$. For this, pick a point $\gamma_s(t)$ close to $q_s$ and let $\sigma_{s,t}$ be the geodesic segment from $p_s$ to $\gamma_s(t)$. Since $d(p_s,q_s)$ is strictly smaller than the injectivity radius, the family $\{\sigma_{s,t}:t\}$ is a variation of $\alpha_s$. Then by first variation formula, $\alpha_s'(q_s)$ is orthogonal to $\gamma_s$ at $q_s$. Taking limit as $s \to 0$, $\alpha_s$ converges to a minimizing geodesic segment $\alpha$ joining $p_2$ and $q_2$ and is orthogonal to $\gamma$ at $q_2$.  Thus $\alpha$ and the two halves of $\gamma$ produce three pairwise distinct minimizing geodesic segments joining $p_2$ and $q_2$. Since $p_2$ and $q_2$ are not conjugate to each other, this contradicts \cite[Chapter 13, Proposition 2.12]{doCarmo92}.

    In conclusion, this contradiction shows that the injectivity radius of $(Y,g_1)$ is at least $\pi$ and finishes the proof of \cref{prop:metric_g1_is_CAT_1}.
\end{proof}

We will now finish the proof of \cref{thm:main}. This theorem is a direct corollary of \cref{thm:main-embedding}. Take $Y=X$, $f$ to be the identity, and apply \cref{thm:main-embedding}. \cref{thm:main} is then immediate.

\section{Examples}

In this section, we construct the example that proves \cref{thm:surface_example}.

\begin{example}
\label{eg:metric_on_surf_rev}

\begin{figure}
    \centering
    \includegraphics[scale=0.9]{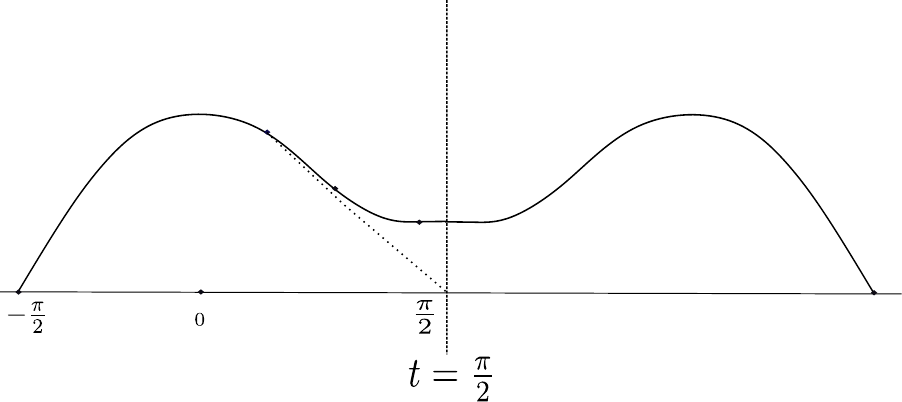}
    \caption{Illustration of the function $f(t)$ used in \cref{eg:metric_on_surf_rev}}
    \label{fig:metric_on_surf_rev}
\end{figure}
Consider a surface of revolution around the $z$-axis of a ($y-z$) planar profile curve $t\mapsto (f(t),g(t))$ parametrized by arclength (i.e. $1=f'(t)^2+g'(t)^2$), so the Gauss curvature is $-f''(t)/f(t)$ (cf. Exercise 7a on p.169 of \cite{doCarmo76}). We wish to choose a positive $f$ such that this is at most $1$. Indeed, in order to find such a function $f$, we modify the $\cos(t)$ function, which is the corresponding first coordinate parameter function for the profile curve $(\cos(t),\sin(t))$ on $[-\frac{\pi}{2},\frac{\pi}{2}]$ for the round sphere. We write $f(t)=\cos(t)+\eps(t)$ where $\eps:[-\frac{\pi}{2},\frac{\pi}{2}]\to [0,1]$ is a smooth function which is identically 0 on $[-\frac{\pi}{2},a]$ for some fixed choice of $0< a < \frac{\pi}{2}$ and satisfies certain other conditions which we will specify shortly. Note that $f'(t)=-\sin(t)+\eps'(t)$ and $f''(t)=-\cos(t)+\eps''(t)$. In order to remain the first coordinate of a unit speed parametrization we require $-1+\sin(t)\leq \eps'(t) \leq 1+\sin(t)$.

The Gauss curvature for the surface of revolution of  this general profile curve is
\[
K_f(t)=-\frac{f''(t)}{f(t)}=\frac{\cos(t)-\eps''(t)}{\cos(t)+\eps(t)}.
\]
We observe that provided $0< \cos(t)+\eps(t)$ and $-\eps''(t)\leq \eps(t)$ for $t\in [a,\frac{\pi}{2}]$, then $K_f\leq 1$. Next, observe that to satisfy both of these conditions, it suffices to choose $\eps$ to be a non-negative and weakly convex function, i.e. $\eps''(t) \geq 0$. We claim that:
\begin{claim}
\label{claim:existence_of_eps_function}
There exists a smooth function $\eps$ defined in a neighborhood of $[a,\frac{\pi}{2}]$ such that:
\begin{itemize}
    \item $0\leq \eps(t)\leq 1$,
    \item $-1+\sin(t)\leq \eps'(t) \leq 1+\sin(t)$, 
    \item $\eps''(t)>0$ for $t\in (a,\frac{\pi}{2})$,
    \item $\eps^{(k)}(a)=0$ for any $k\geq 0$, 
    \item $\eps^{(2k)}(\frac{\pi}{2})=0$, and $\eps^{(2k-1)}(\frac{\pi}{2})=(-1)^{k-1}$ for any for $k\geq 1$ (that is, $\eps^{(k)}(\frac{\pi}{2})$ matches $(-\cos)^{(k)}(t)$ at $t=\frac{\pi}{2}$ for any $k\geq 1$).
\end{itemize}
\end{claim}
 Indeed, to prove this claim, we want a non-negative and weakly convex function $\eps(t)$ that vanishes to all orders at $t=a$ and matches $-\cos(t)$ to all $k$-th order derivatives for $k\geq 1$ at $t=\frac{\pi}{2}$. That this can be done is a consequence of smoothing theory. For the sake of completeness, we will explain the details of this below in the part \textbf{Construction of $\eps(t)$}. 

Now, we have the function $\eps(t)$, and hence $f(t)$, on $[0,\frac{\pi}{2}]$.  We reflect $f(t)$ across the line $t=\frac{\pi}{2}$ in the plane where we graph $(t,f(t))$; see \cref{fig:metric_on_surf_rev}. This finishes the construction of our \cref{eg:metric_on_surf_rev}. The reader should note that the graph in \cref{fig:metric_on_surf_rev} is not the curve being revolved; rather we revolve the curve $t\mapsto (f(t),g(t))$ where $g$ is determined by the unit speed criterion $1=(f')^2+(g')^2$.

\medskip

\noindent \textbf{Construction of $\eps(t)$:}
\label{sec:construct_eps_t}
For the sake of simplicity, we will explain the construction of $\eps$ in a neighborhood of $[0,\frac{\pi}{2}]$ instead of $[a,\frac{\pi}{2}]$. Indeed, by an appropriate linear re-scaling, we can change the point $t=0$ to $t=a$ so that $\eps$ is defined on $[a,\frac{\pi}{2}]$ with all the desired properties remaining unchanged. Moreover, we will also not normalize $\eps(t)$ to take values $[0,1]$; that can be arranged by simply rescaling $\eps(t)$ appropriately. So, to reiterate, we will now construct a function $\eps$ defined in a neighborhood of $[0,\frac{\pi}{2}]$ that satisfies \cref{claim:existence_of_eps_function} with $a=0$.

First, we find a $C^2$ approximation $\eps_0$ of our desired function. To do this, we will define $\eps_0$ by hand on $[\delta, \frac{\pi}{2}-\delta]$ below, see \cref{eqn:defn_eps_0}. Then, we will extend $\eps_0$ 
by the constant $0$ on $[-\delta,\delta]$ and by $(c-\cos t)$ on $[\frac{\pi}{2}-\delta,\frac{\pi}{2}+\delta]$ where $c:=\eps_0(\frac{\pi}{2}-\delta)+\cos(\frac{\pi}{2}-\delta)$. Here we will choose $\delta>0$ to be sufficiently small. It is clear that $\eps_0$ is is convex on $[0,\delta]\cup [\frac{\pi}{2}-\delta,\frac{\pi}{2}]$. So, now we will discuss about $\eps_0$ on the interval $[\delta, \frac{\pi}{2}-\delta]$. On this interval, we take $\eps_0$ to be the following polynomial:  

\begin{align}
\label{eqn:defn_eps_0}
\eps_0(t)=(t-\delta)^3 \frac{12 (\pi-3\delta-t) \cos (\delta)-(\pi -4 \delta) (2\pi-5 \delta-3 t) \sin (\delta)}{3 (\pi -4 \delta)^3}.
\end{align}
It is easy to observe $\eps_0(t) \geq 0$ for $t\in [\delta, \frac{\pi}{2}-\delta ]$, whenever $0<\delta<\frac{\pi}{4}$. The convexity of $\eps_0$ on this interval follows from the explicit computation of the second derivative,
\[
\eps_0''(t)=(t-\delta)\frac{24 (\pi-2 \delta-2 t) \cos (\delta)-4(\pi -4 \delta) (\pi-\delta-3 t) \sin (\delta)}{(\pi -4 \delta)^3}.
\]
Indeed, when $\delta>0$ is sufficiently small, the first term in the sum is almost $48(\frac{\pi}{2}-\delta-t)>0$ while the second term is very close to $0$.

Now, to construct our smooth $\eps(t)$, we convolve $\eps_0$ with a smooth positive bump function $\phi$ whose choice we now explain. Let $\phi_0$ be a smooth bump function supported on $[-\delta,\delta]$ which is even and has $\int_{-\delta}^\delta \phi_0(x)dx=1$. We will now choose a specific normalization of $\phi_0$. For that, let $\alpha_0:=\left( \int_{-\delta}^{\delta}\sin(\frac{\pi}{2}-y)\phi_0(y)dy\right)^{-1}$ and set $\phi:=\alpha_0 \phi_0$. We will work with this normalized bump function. Note that $$\int_{-\delta}^{\delta}\sin(\frac{\pi}{2}-y)\phi(y)dy=1, \text{ while } \int_{-\delta}^{\delta}\cos(\frac{\pi}{2}-y)\phi(y)dy=0.$$ Indeed, the first equality is due to our choice of normalization for $\phi$ while the second equality is a consequence of the fact that $\cos(t)$ is an odd function on $[\frac{\pi}{2}-\delta,\frac{\pi}{2}+\delta]$ and $\phi$ is an even function. 

Now we define $\eps$ by convolving $\eps_0$ with $\phi$,  
\[
\eps(t):=\eps_0 * \phi(t)=\int_{-\infty}^{\infty} \eps_0(y)\phi(t-y)dy=\int_{-\delta}^{\delta} \eps_0(t-y)\phi(y)dy.
\]
Clearly $\eps(t)$ is non-negative and the second expression implies that $\eps''(t)$ remains non-negative as it holds for $\eps_0$. Further, $\eps$ vanishes to all orders at $t=0$ since $\eps_0$ is identically 0 on $[-\delta,\delta]$. Next we observe that all even-order derivatives $\eps^{(2k)}(\frac{\pi}{2})=0$ for $k\geq 1$. Indeed, 
\[
\eps^{(2k)}\left(\frac{\pi}{2}\right)=(-1)^{k+1}\int_{-\delta}^{\delta}\cos(\frac{\pi}{2}-y)\phi(y)dy=0,
\]
for any $k\geq 1$. For the odd-order derivatives, we have 
\[
\eps^{(2k-1)}\left(\frac{\pi}{2}\right)=(-1)^{k-1}\int_{-\delta}^{\delta}\sin(\frac{\pi}{2}-y)\phi(y)dy=(-1)^{k-1},
\]
for $k\geq 1$. So $\eps^{(2k-1)}(\frac{\pi}{2})=(-1)^{k-1}$ which agrees with the odd  order derivatives of $-\cos(t)$ at $t=\frac{\pi}{2}$, as desired. 

\begin{figure}
    \centering
    \includegraphics[scale=0.28]{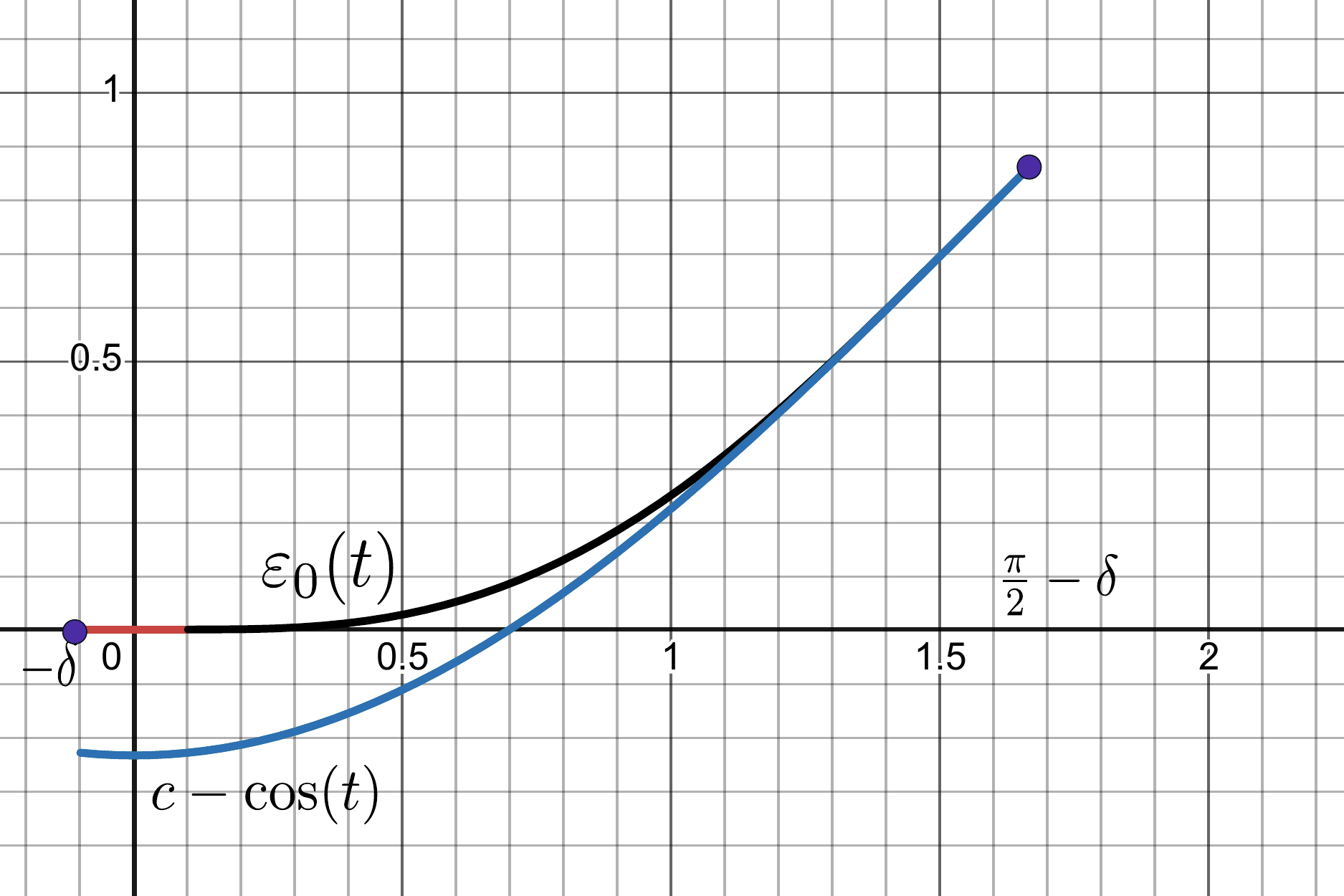}
    \caption{A sample graph of the function $\varepsilon_0(t)$, drawn by choosing $\delta=0.1$. Compare it with $c-\cos(t)$ where $c=\varepsilon_0(\frac{\pi}{2}-\delta)+\cos(\frac{\pi}{2}-\delta)$}
    \label{fig:graph_of_eps0}
\end{figure}

Now the last thing to check is that $-1 \leq \varepsilon'(t)-\sin(t) \leq 1$. In fact, we will show a stronger inequality that $0\leq \varepsilon'(t)\leq \sin(t)$. The lower bound is obvious as $\eps''(t) \geq 0$ and $\eps'(0)=0$.  
Now note that $0\leq \varepsilon_0'(t) \leq \sin (t)$. Indeed, one can compute the derivative of $\eps_0'(t)$ from \cref{eqn:defn_eps_0} and check this by explicit computation. But it is more instructive to simply graph the function $\varepsilon_0(t)$ and note that, by construction, the slope of $\eps_0(t)$ increases faster than the slope of $c-\cos(t)$;  see Figure \ref{fig:graph_of_eps0}. Then,
\begin{align*}
    \varepsilon'(t)&=\int_{-\delta}^{\delta} \eps_0'(t-y)\phi(y)dy \\
    &\leq \int_{-\delta}^{\delta}\sin(t-y)\phi(y)dy =  \frac{1}{2}\left( \int_{-\delta}^{\delta}\sin(t-y)\phi(y)dy + \int_{-\delta}^{\delta}\sin(t+y)\phi(y)dy\right)\\
    & = \int_{-\delta}^{\delta}\sin(t) \cos(y)\phi(y)dy=\sin(t) \int_{-\delta}^{\delta}\sin(\frac{\pi}{2}-y)\phi(y)dy=\sin(t).
\end{align*}

Thus $\eps(t)$ has all our desired properties. 
\end{example}

\begin{remark} \label{rem:minmal} We cannot directly extend the construction of this example to higher dimensions since the symmetry of this construction implies the existence of a minimal $2$-sphere of smaller area than $4\pi$ at the narrowest point of the ``barbell.'' This would contradict the upper curvature condition as we point out in the next paragraph. However, it is unclear to us if a more general construction similar to \cref{eg:metric_on_surf_rev} can be done in higher dimensions.

Note that for any metric on $X$ with a minimal $2$ sphere $S$, the Gauss equations give $K_X(x)=K_S(x)-\la_1^N(x)\la_2^N(x)$ at points of $x\in S$, where $K_X$ is the intrinsic curvature of $X$, $K_S$ is the intrinsic curvature of $S$, and $\la_1^N(x)$ and $\la_2^N(x)$ are the eigenvalues of the second fundamental form of $S$ for any normal direction $N\in (T_xS)^\perp$. By minimality $\la_1^N(x)+\la_2^N(x)=0$ for each $x\in S$ and $N\in (T_xS)^\perp$. Moreover, the Gauss-Bonnet theorem implies $K_S(x)\geq \frac{4\pi}{Area(S)}$ at some point $x\in S$ so there exists some point where $K_X(x)\geq \frac{4\pi}{Area(S)}$. In other words, as soon as there is a minimal 2-sphere smaller in area than an equatorial totally geodesic 2-sphere then $K_X(x)>1$.
\end{remark}

\end{document}